\documentclass[11pt]{amsart}
\usepackage{amssymb}
\usepackage{graphicx}
\usepackage{url}
\usepackage{enumerate}
\usepackage[latin1]{inputenc}
\usepackage{color}

\newcommand\restr[2]{{
  \left.\kern-\nulldelimiterspace 
  #1 
  \vphantom{\big|} 
  \right|_{#2} 
  }}

\newtheorem{theorem}{Theorem}
\newtheorem{proposition}[theorem]{Proposition}
\newtheorem{lemma}[theorem]{Lemma}
\newtheorem{corollary}[theorem]{Corollary}

\theoremstyle{definition}
\newtheorem{definition}[theorem]{Definition}
\newtheorem{remark}[theorem]{Remark}

\newcommand{\RR}{{\mathbb R}}
\newcommand{\SSS}{{\mathbb S}}
\newcommand{\TT}{{\mathbb T}}

\begin{document}

\title[Convexity for torus actions on $b$-symplectic manifolds]{Convexity for Hamiltonian torus actions on $b$-symplectic manifolds}

\author{Victor Guillemin}\address{Department of Mathematics, MIT, Cambridge MA, USA}\email{vwg@math.mit.edu}
\author{Eva Miranda}\address{Department of Mathematics, Universitat Polit\`ecnica de Catalunya, Barcelona, Spain}\email{eva.miranda@upc.edu}\thanks{Eva Miranda is partially supported by the Ministerio de Econom\'{i}a y Competitividad project with reference MTM2015-69135-P}
\author{Ana Rita Pires}\address{Department of Mathematics, Fordham University, New York NY, US / School of Mathematics, Institute for Advanced Study, Princeton NJ, USA}\email{apissarrapires@fordham.edu,apires@ias.edu}\thanks{Ana Rita Pires was partially supported by a  Short Visit Grant from the European Science Foundation's ``Contact and Symplectic Topology'' network, and by the National Science Foundation under agreement number
DMS-1128155. Any opinions, findings and conclusions or recommendations expressed in this material are those of the authors and do not necessarily reflect the views of the National Science Foundation. }
\author{Geoffrey Scott}
\address{Department of Mathematics, University of Toronto, Toronto, Canada}
\email{gscott@math.utoronto.ca}
\date{\today}

\begin{abstract} In \cite{btoric} we proved that the moment map image of a $b$-symplectic toric manifold is a convex $b$-polytope. In this paper we obtain convexity results for the more general case of non-toric hamiltonian torus actions on $b$-symplectic manifolds. The modular weights of the action on the connected components of the exceptional hypersurface play a fundamental role: either they are all zero and the moment map behaves as in classic symplectic case, or they are all nonzero and the moment map behaves as in the toric $b$-symplectic case studied in \cite{btoric}.
\end{abstract}

\maketitle

\section{Introduction}

It is well-known that the moment map image of a hamiltonian torus action on a symplectic manifold is convex \cite{atiyah,guilleminsternberg,guilleminsternberg2}. When the action is toric, the Delzant theorem states that the symplectic toric manifold can be recovered from this moment map image.

The Delzant theorem has been generalized to the context of $b$-symplectic manifolds (also called log-symplectic manifolds), which are even-dimensional Poisson manifolds $(M^{2n}, \Pi)$ for which $\Pi^n$ intersects the zero section of $\Lambda^{2n}TM$ transversally. In this paper, we prove that the moment map image of a hamiltonian torus action on a  $b$-symplectic manifold is convex, even when the dimension of the torus is less than $n$. The proof uses techniques introduced in \cite{btoric} and is strongly inspired by the convexity proof of \cite{atiyah} in the classic symplectic setting. Since our proof uses the classic convexity theorem applied to the symplectic leaves on the singular locus of $\Pi$, we will assume that these leaves are compact (by a result in \cite{guimipi}, it is enough to assume that each component of the singular locus has a compact leaf).

Examples of these torus actions include integrable systems on $b$-symplectic manifolds (cf. \cite{KMS,KM}), and products of a toric $b$-surface (cf. \cite[\S 3]{btoric}) with a symplectic manifold endowed with a Hamiltonian torus action. More generally, consider the construction of a $b$-symplectic toric manifold as a fibration of $b$-symplectic toric surfaces over a standard symplectic toric manifold described in \cite[Remark 39]{btoric}. To obtain instead a $b$-symplectic manifold with a torus action, we can replace the toric base by one with a Hamiltonian action of a torus of smaller dimension, for instance, an almost toric manifold \cite{leung,Karshon,KarshonTolman}.

This paper is organized as follows: in Section \ref{section prelim} we recall the necessary definitions and properties of $b$-symplectic manifolds, in Sections \ref{section nonzero} and \ref{subsec:allzero} respectively we examine the cases of torus actions for which the modular weights of the various connected components of the exceptional hypersurface are all nonzero or all zero. Finally, in Section \ref{sec:noncoexistence} we prove that these two extreme cases are the only possible ones and we conclude with some final remarks on possible further directions in Section \ref{section final}.

\section{Preliminaries}\label{section prelim}

In this section we recall definitions and properties related to $b$-symplectic manifolds. A more detailed exposition can be found in \cite{guimipi,guimipi12}.

Let $(M^{2n},\Pi)$ be a Poisson manifold. If the map
$$p\in M\mapsto(\Pi(p))^n\in\Lambda^{2n}(TM)$$
is transverse to the zero section, then we say that $\Pi$ is a \textbf{$b$-Poisson structure} and that $Z=\{p\in M|(\Pi(p))^n=0\}$
is the \textbf{exceptional hypersurface}. The symplectic foliation of a $b$-Poisson structure has an open symplectic leaf for each component of $M \backslash Z$. The exceptional hypersurface $Z$ is a Poisson submanifold of $(M^{2n}, \Pi)$ and is foliated by symplectic leaves of dimension $2n-2$. There exists a Poisson vector field $v$ on $Z$ transverse to the symplectic foliation. If the symplectic foliation of a component $Z'$ of $Z$ contains a compact leaf $L$, then $Z'$ is the mapping torus of the symplectomorphism $\phi:L\to L$ determined by the flow of $v$.

It is possible to study $b$-Poisson manifolds using symplectic techniques by replacing the tangent and cotangent bundle by the $b$-tangent and $b$-cotangent bundles. A \textbf{$b$-manifold} is a pair $(M, Z)$, where $M$ is an oriented smooth manifold and $Z$ is a closed embedded hypersurface in $M$. A vector field $v$ is a \textbf{$b$-vector field} if $v_p \in T_pZ$ for all $p \in Z$. The $b$-tangent bundle is defined as the vector bundle whose sections are $b$-vector fields, and the \textbf{$b$-cotangent bundle} ${^b}T^*M$ is the dual $({^b}TM)^*$. Sections of $\Lambda^p({^b}T^*M)$ are called \textbf{$b$-forms} and the set of $b$-forms is denoted by ${^b}\Omega^p(M)$. The standard de Rham differential extends to $b$-forms.

A \textbf{$b$-symplectic form} is a closed nondegenerate $b$-form of degree 2. It can be thought of as a smooth symplectic form on $M$ with a mild singularity along $Z$; inverting this singular form gives a $b$-Poisson structure. Thus, $b$-Poisson structures can be seen as symplectic structures modeled over a Lie algebroid (the $b$-tangent bundle). We remind the reader that all $b$-symplectic manifolds in this paper are assumed to have a compact symplectic leaf in each component of their exceptional hypersurface.

Just as we allow our differential forms on a $b$-manifold to have mild singularities, it is also helpful to allow functions to have mild singularities. The sheaf ${^b}C^{\infty}(M)$ of \textbf{$b$-functions} on $M$ consists of functions which are smooth on $M \backslash Z$, and locally around a point on $Z$ can be written in the form $c\,\textrm{log}|y| + g$, where $c \in \mathbb{R}$, $y$ is a local defining function for $Z$, and $g$ is a smooth function.

\begin{definition}
An action of $\mathbb{T}^k$ on a $b$-symplectic manifold $(M,\omega)$ is \textbf{Hamiltonian} if for all $X,Y\in\mathfrak{t}$:
\begin{itemize}
\item the one-form $\iota_{X^\#}\omega$ is exact, i.e., has a primitive $H_X\in\,^bC^\infty(M)$;
\item $\omega(X^\#, Y^\#)=0$.
\end{itemize}
where $X^{\#}$ is the fundamental vector field for the action associated to $X\in \mathfrak{t}$.
\end{definition}

When a $b$-function $f\in C^\infty(M)$ is expressed as $c\log|y|+g$ locally near some point of a component $Z'$ of $Z$, the number $c_{Z'}(f) := c \in \mathbb{R}$ is uniquely determined by $f$, even though the functions $y$ and $g$ are not.

\begin{definition}
Given a Hamiltonian $\mathbb{T}^k$-action on a $b$-symplectic manifold, the \textbf{modular weight} of a connected component $Z'$ of $Z$ is the map $v_{Z'}:\mathfrak{t}\to \mathbb{R}$ given by $v_{Z'}(X)=c_{Z'}(H_X)$. This map is linear and therefore we can regard it as an element of the dual of the Lie algebra $v_{Z'}\in\mathfrak{t}^*$. We denote the kernel of $v_{Z'}$ by $\mathfrak{t}_{Z'}$.
\end{definition}

\begin{definition}
The \textbf{adjacency graph} of a $b$-manifold $(M,Z)$ is a graph $G=(V,E)$  with a vertex for each component of $M\setminus Z$ and an edge for each connected component of $Z$. The endpoints of the edge corresponding to $Z' \subseteq Z$ are the vertices corresponding to the components of $M\setminus Z$ that $Z'$ separates. Multiple edges may have the same pair of endpoints. The \textbf{weighted adjacency graph} of a $b$-symplectic manifold endowed with a hamiltonian $\mathbb{T}^k$-action is $\mathcal{G}=(G,w)$ where $G=(V,E)$ is the adjacency graph of $(M,Z)$ and the weight function $w:E\to\mathfrak{t}^*$ sends each edge to the modular weight of the corresponding component of $Z$.
\end{definition}

A group action is \textbf{effective} if no non-trivial subgroup acts trivially. A \textbf{toric} action is an effective $\mathbb{T}^k$ action on a manifold $M$ such that $k=\frac{1}{2}\dim(M)$. For a toric Hamiltonian action on a $b$-symplectic manifold, the modular weight of each connected component of $Z$ is nonzero (\cite[Claim 13]{btoric}). This will not necessarily be the case for general torus actions that the present paper examines. Indeed, the cases of zero and the nonzero modular weights are very different and we will study them separately.

\begin{remark}\label{rmk:effective} Without loss of generality we will assume that $\mathbb{T}^k$ acts effectively on $M$. Otherwise, we can take the quotient of $\mathbb{T}^k$ by the kernel of its action on $M$ to get an effective action of a quotient group, which will still be a torus \cite[Proposition I.1.7]{audin}.
\end{remark}

\section{The case of nonzero modular weights}\label{section nonzero}

In this section we prove that if the modular weight (corresponding to a Hamiltonian $\mathbb{T}^k$ action) of a connected component of $Z$ is nonzero, then the action and the moment map must ``split'' in a fashion similar to the toric case studied in \cite{btoric}, as detailed in Proposition~\ref{prop:localnonzero}. Indeed, many of the results and proofs in this section follow the results and proofs in that paper.

\begin{proposition}\label{prop:localnonzero}
Let $(M,\omega)$ be a $b$-symplectic manifold endowed with an effective Hamiltonian $\mathbb{T}^k$-action such that the modular weight $v_{Z_i}$ for a connected component $Z_i$ of the exceptional hypersurface is nonzero. Then:
\begin{enumerate}[(a)]
\item for $X$ a representative of a primitive lattice vector of $\mathfrak{t}/\mathfrak{t}_{Z_i}$ that pairs positively with $v_{Z_i}$, we have $\langle X,v_{Z_i}\rangle=c$, where $c$ is the modular period of $Z_i$;

\item $Z_i\simeq \mathcal{L}\times\mathbb{S}^1$, where $\mathcal{L}_i$ is a symplectic leaf of $Z_i$;

\item in a neighborhood of $Z_i$ there is a splitting $\mathfrak{t}\simeq\mathfrak{t}_{Z_i}\times\langle X \rangle$, which induces a splitting $\mathbb{T}^k\simeq\mathbb{T}^{k-1}_{Z_i}\times\mathbb{S}^1$. The $\mathbb{T}^{k-1}_{Z_i}$-action on $Z_i$ induces a Hamiltonian $\mathbb{T}^{k-1}_{Z_i}$-action on $\mathcal{L}_i$, let $\mu_{\mathcal{L}_i}:\mathcal{L}_i\to\mathfrak{t}^*_{Z_i}$ be its moment map;
\item\label{theonethatineed} there is a neighborhood $\mathcal{L}_i\times\mathbb{S}^1\times(-\varepsilon,\varepsilon)\simeq\mathcal{U}\subset M$ of $Z_i$ such that the $(\mathbb{T}^{k-1}_{Z_i}\times\mathbb{S}^1)$-action on $\mathcal{U}\setminus Z_i$ is given by $(g,\theta)\cdot(\ell,\rho,t)=(g\cdot\ell,\rho+\theta,t)$ and has moment map
\begin{equation*}\begin{aligned}
\mu_{\mathcal{U}\setminus Z_i}:\mathcal{L}_i\times\mathbb{S}^1\times((-\varepsilon,\varepsilon)\setminus\{0\})&\to\mathfrak{t}^*\simeq\mathfrak{t}^*_{Z_i}\times\mathbb{R}\\
(\ell,\rho,t)&\mapsto(\mu_{\mathcal{L}_i}(\ell),c\log|t|).
\end{aligned}\end{equation*}
\end{enumerate}
\end{proposition}

\begin{proof}
In light of Remark~\ref{rmk:effective}, the proofs of \cite[Proposition 15, Corollary 16, Lemma 17, Proposition 18]{btoric} for the toric case hold in this context with the obvious changes.
\end{proof}

From the splitting in the local model we conclude, as in the toric case, that the moment image of $\mathcal{U}\setminus Z_i$ where $\mathcal{U}$ is a neighborhood of a connected component of $Z_i$ is of the form $\Delta_{\mathcal{L}_i}\times (-\infty,N)$ for some $N\in\mathbb{R}$, where the convex polytope $\Delta_\mathcal{L}\subset\mathfrak{t}^*_{Z_i}$ is the image of $\mu_{\mathcal{L}_i}$.

\begin{proposition}
Let $(M,\omega)$ be a $b$-symplectic manifold endowed with an effective Hamiltonian $\mathbb{T}^k$-action and let $Z_1$ and $Z_2$ be two connected components of the exceptional hypersurface $Z$ which correspond to edges incident on the same vertex of the adjacency graph of $(M,\omega)$. If the modular weights of $Z_1$ and $Z_2$ are both nonzero, then they are a negative scalar multiple of each other, and there are no other edges incident to that vertex on the adjacency graph. Therefore $\mathfrak{t}_{Z_1}=\mathfrak{t}_{Z_2}$.

In particular, if the modular weights of all connected components of $Z$ are nonzero, then the adjacency graph of $M$ is either a line with any number of vertices or a circle with an even number of vertices and in both cases $\mathfrak{t}_Z:=\mathfrak{t}_{Z_i} \,\,\forall_i$ is well-defined.
\end{proposition}

\begin{proof}
The discussion and proofs from \cite[Section 4.2]{btoric} hold.
\end{proof}

This allows us to define a moment map for a Hamiltonian torus action on a $b$-symplectic manifold $(M, \omega)$ in the case when the modular weight corresponding to each component of $Z$ is nonzero. Its codomain is the $b$-manifold $(\mathcal{R}_\mathcal{G},\mathcal{Z}_\mathcal{G})$, where $\mathcal{G}=(G=(V,E),w)$ is the weighted adjacency graph, and
\begin{align*}
\mathcal{R}_\mathcal{G} &= \mathfrak{t}^* \times V \sqcup \mathfrak{t}_Z^* \times E\\
\mathcal{Z}_\mathcal{G} &= \hspace{1.5cm} \mathfrak{t}_Z^* \times E.
\end{align*}
This space can be endowed with a smooth structure that depends on the weight function $w$, the technical details are the same as in \cite[Section 5]{btoric}.

\begin{definition}
Let $(M,\omega)$ be a $b$-symplectic manifold endowed with an effective Hamiltonian $\mathbb{T}^k$-action such that that the modular weights for the connected components of the exceptional hypersurface are all nonzero. A smooth $\mathbb{T}^k$-invariant $b$-map $\mu:M\to\mathcal{R}_\mathcal{G}$ is a \textbf{$b$-moment map} for the action if the map $\mathfrak{t}\ni X\mapsto \mu^X\in C^\infty(M)$ given by $\mu^X(p):=\langle \text{pr}_1\circ\mu(p),X\rangle$ is linear and
$$\iota_{X^\#}\omega=d\mu^X,$$
where $X^\#$ is the vector field on $M$ generated by $X$, where $\textrm{pr}_1$ is the projection $\mathcal{R}_\mathcal{G} \backslash \mathcal{Z}_\mathcal{G} \rightarrow \mathfrak{t}^*$.
\end{definition}

We recover from \cite[Definitions 29 and 30]{btoric} the notion of a \textbf{$b$-polytope} in $\mathcal{R}_\mathcal{G}$: it is a bounded subset $P$ that intersects each component of $Z_\mathcal{G}$ and can be expressed as a finite intersection of half-spaces. Note that there are two types of hyperplanes in $\mathcal{R}_\mathcal{G}$ which divide $\mathcal{R}_\mathcal{G}$ into two connected components and therefore give rise to (two) half-spaces: hyperplanes that are perpendicular to a vector in $\mathfrak{t}\setminus\mathfrak{t}_Z$ and are completely contained in $\mathfrak{t}^*\times\{v\}$ for some vertex $v$ of $G$, and hyperplanes which are perpendicular to a vector in $\mathfrak{t}_Z$ and intersect each component of $\mathcal{Z}_\mathcal{G}$.

\begin{proposition}
Let $(M,\omega)$ be a $b$-symplectic manifold endowed with an effective Hamiltonian $\mathbb{T}^k$-action such that that the modular weights for the connected components of the exceptional hypersurface are \underline{all nonzero}, and let $\mathcal{G}$ be its weighted adjacency graph. Then:
\begin{enumerate}[(a)]
\item there exists a $b$-moment map $\mu:M\to\mathcal{R}_\mathcal{G}$;
\item the image of this map is a $b$-polytope in $\mathcal{R}_\mathcal{G}$;
\item the vertices of the $b$-polytope are precisely the fixed points of the $\mathbb{T}^k$-action.
\end{enumerate}
\end{proposition}

\begin{proof}
\begin{enumerate}[(a)]
\item The proofs of \cite[Proposition 26 and Theorem 27]{btoric} hold.

\item Using the local model of Proposition~\ref{prop:localnonzero}(\ref{theonethatineed}) we see that locally near $\mu(Z)$ the moment image is a $b$-polytope. By performing symplectic cutting on $M$ arbitrarily close to each connected component of $Z$ we can construct a collection of compact connected symplectic manifolds, whose moment images are, by the convexity results for classic symplectic manifolds, convex polytopes. We note that symplectic cutting is possible: it can be done using the component of the moment map which in the local coordinates of Proposition~\ref{prop:localnonzero}(\ref{theonethatineed}) is of the form $c\log|t|$ (recall that $t$ is a defining function for that particular  connected component of $Z$). The image $\mu(M)$ is the union of these polytopes with the $b$-polytopal neighborhoods of $\mu(Z)$, which produces a $b$-polytope. As the location of the symplectic cuts approaches $Z$, the polytopes enlarge; the limit is a polytope with recession cone equal to the modular weight vectors corresponding to the adjacent components of $Z$ (this can be seen using the description of a moment map near $Z$ from Proposition~\ref{prop:localnonzero}(\ref{theonethatineed}). This shows that the moment map polytopes obtained from the symplectic cuts on the complement of $Z$ glue consistently with the product-type moment map polytopes obtained from neighborhoods of the connected components of $Z$. The result can be described as a finite union of half-spaces in the following way. First, each half-space that defines the polytope $\Delta_{\mathcal{L}}\subset\mathfrak{t}^*_{Z}$ (the polytope ``at infinity''), is described by a perpendicular vector in $\mathfrak{t}_Z$ and a scalar; this same data defines a half-space in $\mathcal{R}_\mathcal{G}$ of the type whose boundary hyperplane intersects every component of $\mathcal{Z}_\mathcal{G}$. Second, each half-space that defines the limit polytopes in $\mathfrak{t}^*$ also define (using the same data) a half-space in $\mathfrak{t}^*\times\{v\}$ for the $v$ corresponding to the component of $M \backslash Z$ described by this limit polytope. Taken together, these half-spaces define the moment map image in $\mathcal{R}_\mathcal{G}$.


\item By the local form for the action in Proposition~\ref{prop:localnonzero}(\ref{theonethatineed}), no point in $Z$ is fixed by the action. The fixed points are therefore contained in $M\setminus Z$, and again by performing symplectic cutting on $M$ arbitrarily close to each connected component of $Z$ we can construct a collection of compact connected symplectic manifolds $M_{i,\varepsilon}$, one for each edge of the adjacency graph, whose union contains a copy of a neighborhood of each of the fixed points of $M$. By the classic convexity results, the fixed points map to vertices of the moment polytopes of the symplectic manifolds $M_{i,\varepsilon}$. When we take the union of the moment polytopes of the $M_{i,\varepsilon}$ with the $b$-polytopal neighborhoods of $\mu(Z)$ we conclude that the fixed points of $M$ map exactly to the vertices of $\mu(M)$.
\end{enumerate}
\end{proof}

\section{The case of zero modular weights}\label{subsec:allzero}

We now study the Hamiltonian torus actions for which the modular weights of the connected components of $Z$ are \emph{all} zero. This means that for any $X\in\mathfrak{t}$ the $b$-form $\iota_{X^\#}\omega$ is an honest de Rham form and has a primitive $H_X$ which is a smooth function. Because each $X^\#$ is a Hamiltonian vector field on $M$ with respect to the Poisson structure given by inverting the $b$-symplectic form, each $X^{\#}$ is tangent to the symplectic leaves in $Z$. In this case, there is a well-defined map $\mu:M\to\mathfrak{t}^*$ given by $\mu(p)(X)=H_X(p)$, so we extend the definition for the classic symplectic case and call it a moment map for the action.

We begin this section by proving some facts about effective circle actions with zero modular weights (i.e., with a smooth Hamiltonian), following \cite{atiyah}'s argument for the symplectic case, and then conclude with the consequences for effective $\TT^k$-actions with zero modular weights on $b$-symplectic manifolds.

\subsection{Effective circle actions on $b$-symplectic manifolds with smooth Hamiltonians}

\begin{lemma}\label{lemma1}
Let $(M, \omega)$ be a $b$-symplectic manifold endowed with an effective $\SSS^1$-action generated by the $b$-vector field $X^{\#}$, and let $Z$ be the exceptional hypersurface.

If $p\in Z$ is fixed by the action, then $X^\#$ vanishes at $p$ as a section of ${^b}TM$. Moreover, if $H_X$ is a smooth Hamiltonian for $X^\#$, then $dH_X$ vanishes at $p$ as a section of $T^*M$.

\end{lemma}

\begin{remark}
By Lemma \ref{lemma1}, the fixed points of a Hamiltonian $\SSS^1$-action with smooth Hamiltonian $H_X$ are exactly the points where $dH_X$ vanishes as a smooth form, as in classic symplectic geometry. However, unlike in classic symplectic geometry, the same does not hold for $\RR$-actions. For example, consider the $b$-symplectic manifold $(\mathbb{R}^2, \frac{1}{x}dx\wedge dy)$ endowed with the $\RR$-action generated by the $b$-vector field $X^\#=x\frac{\partial}{\partial y}$, which fixes all points of the exceptional hypersurface $Z=\{0\}\times\RR$. The Hamiltonian is the smooth function $H_X(x,y)=x$, whose differential $dx$ never vanishes, and in particular doesn't vanish at the fixed locus $Z$.
\end{remark}

\begin{proof}

In order to prove the first assertion, take local coordinates $\bar{x}=(x_1, \dots, x_n)$ around $p$ such that $Z=\{\bar{x}:x_1 = 0\}$. Then the $b$-vector field can be written as
\[
X^\# = a_1x_1\frac{\partial}{\partial x_1} + \sum_{i = 2}^n a_i\frac{\partial}{\partial x_i}
\]
for smooth functions $a_i=a_i(\bar{x})$, with $a_i(\bar{0}) = 0$ for $i \geq 2$. Assume towards a contradiction that $a_1(\bar{0}) \neq 0$, and without loss of generality that $a_1(\bar{0}) > 0$. Let $U = \{\bar{x}:a_1(\bar{x}) > 0\}$, and $U'$ be a smaller open set containing $\bar 0$ such that the time-1 flow (assume that $\SSS^1$ is parametrized from $0$ to $1$) of any point in $U'$ remains in $U$. Then for any point in $U' \cap \{\bar x:x_1 > 0\}$, the time-1 flow of the trajectory has monotonically increasing $x_1$-coordinate, and therefore the trajectory cannot be periodic, which yields a contradiction with it coming from an $\SSS^1$-action.

Now consider the case that there exists a smooth Hamiltonian $H_X$ for the action. Then by the argument above $X^\#$ vanishes at $p$ as a $b$-vector field, so $dH_X$ vanishes at $p$ as a $b$-form, i.e., $dH_X(v) = 0$ for all $v \in T_pZ$. This implies that the kernel of $dH_X\in\Omega^1(M)$ (as a smooth form) at $p$ contains the codimension-1 subspace $T_pZ$. By \cite[Lemma 17]{btoric}, it follows that $dH_X$ vanishes at $p$ (as a smooth form).
\end{proof}

\begin{lemma} \label{lemma2}
Let $(M,\omega)$ be a $b$-symplectic manifold endowed with an effective $\SSS^1$-action generated by the $b$-vector field $X^{\#}$ and with a smooth Hamiltonian $H_X$. Let $p$ be a fixed point on the exceptional hypersurface $Z$, and $\mathcal{L}$ be the symplectic leaf containing $p$.

Then there exist vectors $v_Z \in T_pZ \backslash T_p\mathcal{L}$ and $v_M \in T_pM \backslash T_pZ$ that are fixed by the isotropy representation on $T_pM$.
\end{lemma}

\begin{proof} Let $\phi_t$ be the time-$t$ flow of the $\SSS^1$ action. Choose an $\SSS^1$-invariant metric on $M$, and let $v_M$ be any nonzero element of $T_pZ^{\perp}$ and $v_Z$ be any element of $T_p\mathcal{L}^\perp \cap T_pZ$. By the invariance of the metric, and because $T_p\mathcal{L}$ and $T_pZ$ are $d\phi_t$-invariant subspaces, the span of $v_Z$ in $T_pM$ and the span of $v_M$ in $T_pM$ are both invariant subspaces. Because the only representation of $\SSS^1$ on a one-dimensional real subspace is the identity representation, it follows that $v_Z$ and $v_M$ are both fixed by the isotropy representation $d\phi_t$.
\end{proof}

\begin{proposition}\label{prop:even}
Let $(M,\omega)$ be a $b$-symplectic manifold endowed with an effective $\SSS^1$-action generated by the $b$-vector field $X^{\#}$ and with a smooth Hamiltonian $H_X$.
Then $H_X$ is Morse-Bott and all its indices and coindices are even.

More precisely, for points $p \in Z$ where $dH_X(p) = 0$, the indices and coindices of $\textrm{Hess}_p(H_X)$ are the same as those of $\textrm{Hess}_p(\restr{H_X}{\mathcal{L}})$, where $\mathcal{L}$ is the symplectic leaf through $p$.
 \end{proposition}

 \begin{proof} We need to show that the set $Y := \{p \in M \mid dH_X(p) = 0\}$ is a smooth manifold, that for each such $p \in Y$, the tangent space $T_pY$ coincides with the kernel of $\textrm{Hess}_p(H_X)$, and that all indices and coindices are even. Away from $Z$ we are in a symplectic manifold and therefore these facts hold \cite[Lemma 2.2]{atiyah}, so we are left with proving them for points $p \in Z \cap Y$.

By Lemma \ref{lemma1}, the fixed points of the action on $Z$ are exactly the points of $Z\cap Y$. Let $p \in Z \cap Y$, and consider an invariant metric in a neighborhood of $p$. Then the exponential map (with respect to this metric) gives a local diffeomorphism between the fixed subspace of the tangent space and the set of fixed points of the action, so $Y$ is indeed (locally, and therefore globally) a submanifold.

Let $v \in T_pY \subseteq T_pM$, and pick an extension $\tilde{v}$ of $v$ to a neighborhood of $p$ such that $\tilde{v}$ is tangent to $Y$. Because the tangent space of $Y$ coincides with the kernel of the hessian of $H_X$ for points of $Y$ away from $Z$, i.e.,
 $$\tilde{v}_q \in \textrm{ker}\left(\textrm{Hess}_q(H_X)\right) \text{ for all }q \in Y \backslash (Y \cap Z),$$
by continuity and because $Y$ is not a subset of $Z$ (in Lemma \ref{lemma2}, the existence of $v_M$ proves this) it follows that $v \in \textrm{ker}(\textrm{Hess}_p(f))$.

If we choose a basis for $T_pM$ so that the first $n-2$ vectors are a basis for $T_p\mathcal{L}$, and the last two vectors are $v_Z$ and $v_M$ (from Lemma \ref{lemma2}), the argument in the paragraph above implies that the Hessian matrix with respect to this basis will be
\[
\left[ \begin{array}{c c c} \textrm{Hess}_p(\restr{H_X}{\mathcal{L}}) & 0 & 0 \\ 0 & 0 & 0 \\ 0 & 0 & 0 \end{array}\right]
\]
which proves that the index and coindex of $\textrm{Hess}_p(H_X)$ equals the index and coindex of $\textrm{Hess}_p(\restr{H_X}{\mathcal{L}})$, each of which must be even by the classic argument for Hamiltonian $\SSS^1$-actions on compact symplectic manifolds (considering $\restr{H_X}{\mathcal{L}}$ as Hamiltonian on $(\mathcal{L}, \omega_{\mathcal{L}})$, where $\omega_{\mathcal{L}}$ is the symplectic form on $\mathcal{L}$ induced by the Poisson structure $\omega^{-1}$  \cite[Lemma 2.2]{atiyah}.
 \end{proof}

\begin{remark}\label{rmk:not one}
\cite[Lemma 5.51]{mcduff}
The level sets of a Morse-Bott function with no indices or coindices equal to 1 are connected.
\end{remark}

\begin{proposition}\label{prop:the thing}
Let $(M,\omega)$ be a connected $b$-symplectic manifold endowed with an effective $\SSS^1$-action generated by the $b$-vector field $X^{\#}$ and with a smooth Hamiltonian $H_X$. Let $\mathcal{L}$ be a symplectic leaf of one component of the exceptional hypersurface $Z$. Then,
$$H_X(\mathcal{L}) = H_X(M).$$
\end{proposition}

\begin{proof}
Clearly $H_X(\mathcal{L}) \subseteq H_X(M)$, and because we are assuming that the symplectic leaves are compact, $H_X(\mathcal{L})$ is an interval in $\RR$, say $[a, b]$. By Proposition~\ref{prop:even} and Remark~\ref{rmk:not one}, the level set $H_X^{-1}(b)$ is connected and contains a point $p \in \mathcal{L}$ which is a fixed point of the $\SSS^1$-action. By Lemma~\ref{lemma1}, this is a critical point of $H_X$, so $H_X^{-1}(b)$ must contain this critical submanifold of the Morse-Bott function.

The coindex of $\textrm{Hess}_p(H_X)$ equals the coindex of the $\textrm{Hess}_p(\restr{H_X}{\mathcal{L}})$, which is zero because $b$ is a global maximum of the restriction of $H_X$ to $\mathcal{L}$. Then, the component of $H_X^{-1}(b)$ containing $p$ is a local maximum, which must be a global maximum since $H_X^{-1}(b)$ is connected.
\end{proof}

\subsection{Effective $\TT^k$-actions with zero modular weights}

A moment map for an effective $\TT^k$-action with zero modular weights is a collection of the smooth Hamiltonians for the actions of its coordinate circles. Then, as a direct consequence of Proposition~\ref{prop:the thing} we have:

\begin{corollary}\label{prop:mu(M)=mu(Z)}
Let $(M,\omega)$ be a $b$-symplectic manifold endowed with an effective Hamiltonian $\mathbb{T}^k$-action such that that the modular weights for the connected components of the exceptional hypersurface are \underline{all zero}.

Then there exists a moment map $\mu:M\to\mathfrak{t}^*$, its image is the convex hull of the image of the fixed point set, and more precisely
$$\mu(M)=\mu_{|Z}(Z)=\mu_{|\mathcal{L}}(\mathcal{L}),$$
where $\mathcal{L}$ is a symplectic leaf of any of the connected components of the exceptional hypersurface $Z$.
\end{corollary}

\section{Why zero and nonzero modular weights cannot coexist}\label{sec:noncoexistence}

In the previous sections we have obtained convexity results for two extreme cases: when the modular weights of the connected components of $Z$ are either all zero or all nonzero. In this section we show that only these extreme cases occur.

\begin{theorem}
Let $(M,\omega)$ be a connected $b$-symplectic manifold endowed with an effective Hamiltonian $\mathbb{T}^k$-action.
Then either the modular weights for the connected components of the exceptional hypersurface $Z$ are all zero, or they are all nonzero.
\end{theorem}

\begin{proof}
Assume towards a contradiction that there exist $Z_1$ and $Z_2$ two connected components of $Z$ such that  the modular weight of $Z_1$ is zero and the modular weight of $Z_2$ is $v_{Z_2} \neq 0$. Take any element $X \in \mathfrak{t}$ in the period lattice that pairs positively with $v_{Z_2}$ and pick a Hamiltonian $b$-function $H_X$ for the vector field $X^\#$. By definition of modular weight, this function will equal $-\infty$ at $Z_2$. But we can take a symplectic cut at an arbitrarily large negative number $-N$, and the $\SSS^1$-action given by $X^\#$ on the symplectic cut $M_{\geq -N}$ is now an action by a smooth Hamiltonian function, and therefore $H_X(\mathcal{L}) = H_X(M_{\geq -N})$ where $\mathcal{L}$ is a leaf inside $Z_1$. If we choose $-N$ smaller than the minimum of $H_X(\mathcal{L})$, this gives a contradiction.
\end{proof}

\section{Final remarks}\label{section final}

\subsection{$\mathbb T^k$-actions on generalizations of  $b$-symplectic manifolds}

We can generalize the definition of $b$-symplectic manifold by allowing the exceptional hypersurface $Z$ to have transverse self-intersections, instead of insisting that it is embedded. In keeping with the spirit of \cite{Melrose}, we  refer to these structures as $c$-symplectic -- $b$ for boundary, $c$ for corner. The simplest example of a $c$-symplectic manifold is the product of two $b$-symplectic manifolds, see \cite{gualtierietal} for more examples under the generic notation of log symplectic (for both $b$- and $c$-symplectic manifolds).

We have proved that for $b$-symplectic manifolds, zero and nonzero modular weights cannot coexist. To examine the $c$-symplectic case, consider the $c$-symplectic manifold $(M=\mathbb T_1^2\times\mathbb T_2^2,\omega_1+\omega_2)$ which is the product of two $b$-symplectic tori: $(\mathbb{T}_1^2,\omega_1=\frac{1}{\sin\theta_1}d\theta_1\wedge d\alpha_1)$ and $(\mathbb{T}_2^2,\omega_2=\frac{1}{\cos\theta_2}d\theta_2\wedge d\alpha_2)$.
The exceptional hypersurface of $(\mathbb{T}_1^2,\omega_1)$ consists of the union of the two circles $Z_1=\left\{\theta_1=0,\pi\right\}$, as does that of $(\mathbb{T}_2^2,\omega_2)$, with $Z_2=\left\{ \theta_2=\frac{\pi}{2},\frac{3\pi}{2}\right\}$. The exceptional hypersurface $Z$ of $M$ is the union $(Z_1\times\mathbb{T}^2_2)\cup(\mathbb{T}_1^2\times Z_2)$. This set has a stratification given by the rank of the Poisson structure: one stratum (the lowest-dimensional one) is the intersection $S_c=(Z_1\times\mathbb{T}^2_2)\cap(\mathbb{T}_1^2\times Z_2)$ and the other is the symmetric difference $(Z_1\times\mathbb{T}^2_2)\Delta(\mathbb{T}_1^2\times Z_2)$, where the $c$-symplectic structure is actually $b$-symplectic.

Consider the circle action on $M$ generated by $\frac{\partial}{\partial \alpha_1}$: its modular weight would be nonzero for the set $(Z_1\times\mathbb{T}^2_2)\setminus S_c$ but zero for the set $(\mathbb{T}_1^2\times Z_2)\setminus S_c$.

\subsection{ Deficiency-one torus actions and circle actions on four dimensional $b$-symplectic manifolds}

One could try to generalize the tools developed in \cite{KarshonTolman} to treat periodic  hamiltonian flows on $4$-dimensional $b$-symplectic manifolds. Possible applications would be to adapt the classification of $\SSS^1$-actions on $4$-symplectic manifolds \cite{Karshon} to the $b$-case, and to study when torus actions on $b$-symplectic manifolds extend to toric actions. In particular, it would be interesting to determine if there exists a $b$-analogue of the following result:

\begin{theorem} \cite{Karshon}
Every $4$-dimensional compact Hamiltonian $\SSS^1$-space with isolated fixed points comes from a K{\"a}hler toric manifold by restricting the action to a sub-circle.
\end{theorem}

\subsection{Reduction and Duistermaat-Heckman theory}
In \cite{btoric} we describe the moment $b$-polytope as a fibration by one-dimensional $b$-polytopes
 ($b$-lines or $b$-circles, moment images of $b$-Hamiltonian $\mathbb{S}^1$-spaces) over the moment polytope $\Delta_\mathcal{L}$ of a leaf of $Z$.

In the general case of a $\mathbb T^k$-action  with non-zero  modular weights\footnote{When the  modular weights are zero this study reduces to the symplectic one in the codimension-one symplectic foliation of the exceptional hypersurface $Z$.}  we can also apply the reduction scheme with respect to a $\mathbb T^{k-1}$-action.  The corresponding $b$-Hamiltonian $\mathbb{S}^1$-spaces may vary depending on which \emph{chamber} of $\Delta$ they sit over -- a chamber is a connected component of $\Delta$ determined by regular values of the moment map. Observe that in the toric case there is only one such chamber, which is why all the fibers are of the same type: either all $b$-spheres or all $b$-tori. One could apply a similar scheme to the case of general torus actions and relate the diffeomorphism type of the fibers with the different types of chambers.



\end{document}